\documentclass{amsart}
\usepackage{graphicx}
\usepackage{amsmath, amssymb, latexsym}
\title{Distinguishing partitions of complete multipartite graphs}
\author{Michael Goff}

\newtheorem{theorem}{Theorem}[section]

\newtheorem{corollary}[theorem]{Corollary}
\newtheorem{lemma}[theorem]{Lemma}
\newtheorem{definition}[theorem]{Definition}

\def\proof{\smallskip\noindent {\it Proof: \ }}
\def\proofof#1{\smallskip\noindent {\it Proof of #1: \ }}
\def\endproof{\hfill$\square$\medskip}

\begin{document}

\begin{abstract}
A \textit{distinguishing partition} of a group $X$ with automorphism group $\mbox{aut}(X)$ is a partition of $X$ that is fixed by no nontrivial element of $\mbox{aut}(X)$.  In the event that $X$ is a complete multipartite graph with its automorphism group, the existence of a distinguishing partition is equivalent to the existence of an asymmetric hypergraph with prescribed edge sizes.  An asymptotic result is proven on the existence of a distinguishing partition when $X$ is a complete multipartite graph with $m_1$ parts of size $n_1$ and $m_2$ parts of size $n_2$ for small $n_1$, $m_2$ and large $m_1$, $n_2$.  A key tool in making the estimate is counting the number of trees of particular classes.
\end{abstract}

\date{November 25, 2012}

\maketitle

\section{Introduction}
\label{introsection}

The distinguishing partition problem asks, given a finite set $X$ with a group $\mbox{aut}(X)$ that acts on $X$, whether there exists a partition $P$ of the elements of $X$ such that no nontrivial element of $\mbox{aut}(X)$ fixes $P$.  Formally, consider a partition $P = \{P_1, \ldots, P_t\}$ and $\gamma \in \mbox{aut}(X)$.  For general $X' = \{x_1, \ldots, x_i\} \subset X$, let $\gamma(X') = \{\gamma(x_1),\ldots,\gamma(x_i)\}$.  Then let $\gamma(P) = \{\gamma(P_1),\ldots,\gamma(P_t)\}$.  We say that $P$ is a distinguishing partition if $\gamma(P) \neq P$ for all nontrivial $\gamma \in \mbox{aut}(X)$.

Not all sets $X$ with group action $\mbox{aut}(X)$ have a distinguishing partition.  For example, if $\mbox{aut}(X)$ is the group of all permutation on $X$ and $|X| \geq 2$, then $X$ does not have a distinguishing partition.  Conversely, if $\mbox{aut}(X)$ is the trivial group, then all partitions of $X$ are distinguishing.  In general, the conditions for the existence of a distinguishing partition can be quite complex, even in a relatively restricted setting such as taking $X$ to be a complete multipartite graph, acted upon by its automorphism group.  Informally, the difficulty is that if a partition $P$ consists of few large parts, then a nontrivial automorphism might fix each part, while if $P$ consists of many small parts, then a nontrivial automorphism might permute the parts.

Ellingham and Schroeder \cite{distpart} first considered the distinguishing partitions problem for complete equipartite graphs.  Their finding is that if $X$ is a complete equipartite graph with $m$ parts, each of size $n$, then $X$ has a distinguishing partition if and only if $m \geq f(n)$ for $f(2) = f(14) = 6$, $f(6) = 5$, and otherwise $f(n) = \lfloor \log_2(n+1) \rfloor + 2$.  In this setting, $\mbox{aut}(X)$ is the wreath product action $S_n\ \mbox{Wr}\ S_m$.

The distinguishing partition is a measure of the level of symmetry of a group action, and as such the concept is closely related to the well-studied distinguishing number, as introduced by Albertson and Collins \cite{albertsoncollins} on a graph and by Tymoczko \cite{tymoczko} for a general group action.  Other such measures are the cost of $2$-distinguishing \cite{cost2} and the determining set \cite{detset}.  The survey of Bailey and Cameron \cite{baileycameron} shows how these concepts have appeared independently in many different settings.

For the remainder of this paper, we will consider the case that $X$ is a complete multipartite graph with its automorphism group.  We denote by $X = K_{n_1, \ldots, n_m}$ the complete multipartite graph with maximal independent sets $X_i$ of size $n_i$ for $1 \leq i \leq m$.  Also, $K_{m_1(n_1), m_2(n_2)}$ denotes the complete multipartite graph with $m_i$ parts each of size $n_i$  for $i = 1,2$.  We focus in particular on $K_{m_1(n_1),m_2(n_2)}$ for fixed $n_1$ and $m_2$ and large $m_1$ and $n_2$.

Based on the results of Ellingham and Schroeder \cite{distpart}, we might expect a complete multipartite graph to have a distinguishing partition if it has many small parts, and not to have a distinguishing partition if it has few large parts.  In our setting, which combines these two extremes, it seems natural to expect that a distinguishing partition, in the asymptotic sense, would exist if $n_2/m_1$ does not exceed a certain ratio.  Our main result is that this is indeed the case.

\begin{theorem}
\label{maintheorem}
Fix $n_1 \geq 2$ and $m_2 \geq 1$, and suppose that $m_1$ is large.  There exists a value $r = r_{n_1,m_2}$ such that the following holds.  $K_{m_1(n_1), m_2(n_2)}$ has a distinguishing partition if and only if $$n_2 \leq r_{n_1, m_2}m_1 + \epsilon(m_1)$$ for some function $\epsilon(m_1) \in o(m_1)$.
\end{theorem}

We have that $r_{2,m_2} = 1$.  For $n_1 \geq 3$, we define $r_{n_1,m_2}$ by first choosing values of $j = j_{n_1,m_2}$ and $k = k_{n_1,m_2}$ such that
$$n_1 = 2 + {m_2 \choose 0} + {m_2 \choose 1} + \cdots + {m_2 \choose j} + k,$$ 
with either 
$$j < \lfloor (m_2-1)/2 \rfloor \quad \mbox{and} \quad 0 \leq k < {m_2 \choose j+1}, \quad \mbox{or} \quad j = \lfloor (m_2-1)/2 \rfloor \quad \mbox{and} \quad k \geq 0.$$ 
If $j < \lfloor (m_2-1)/2 \rfloor$, then let 
$$r = 1 + \sum_{i=0}^{j}\frac{m_2-i}{m_2}{m_2 \choose i} + \frac{m_2 - j - 1}{m_2}k,$$ 
and otherwise choose 
$$r = 1 + \sum_{i=0}^{j}\frac{m_2-i}{m_2}{m_2 \choose i} + \frac{1}{2}k.$$  We say that $j_{2,m_2} = -1$.

The structure of the paper and the proof Theorem \ref{maintheorem} is as follows.  In Section \ref{treesection}, we establish basic concepts on enriched trees and hypergraphs which are used heavily throughout the proof.  In Section \ref{corrsection}, we show how a type of partition of $K_{m_1(n_1), m_2(n_2)}$ known as a regular partition may be represented as a hypergraph with $m_i$ edges of size $n_i$, $i=1,2$.   We establish key lemmas for the general result in Section \ref{generallemmas}.  In Section \ref{modelconst}, we provide the general construction that maximizes $n_2$ to within an additive constant, given $m_1, n_1, m_2$.  Then we prove that for large $m_1$ relative to $n_1$ and $m_2$, if $n_2' > n_2$ and $K_{m_1(n_1),m_2(n_2')}$ has a distinguishing partition, then so does $K_{m_1(n_1),m_2(n_2)}$.

In Section \ref{n1is2}, we focus on the case that $n_1 = 2$.  Then the following refinement of Theorem \ref{maintheorem} holds.
\begin{theorem}
\label{n1is2thm}
There exist constants $\alpha > 0$ and $\beta > 1$ and
$$z := \left\lfloor {\log_{\beta}\left(\frac{m_1(\beta-1)}{\alpha \beta}\left({\log_\beta m_1}\right)^{3/2}\right)} \right\rfloor$$
such that Theorem \ref{maintheorem} holds with $\epsilon(m_1)$ of the form $$\frac{m_1}{z+1} + (1+o_{m_1}(1))\alpha \beta^z z^{-7/2}\left(\frac{\beta}{\beta-1}\right)^2 \approx \frac{m_1}{\log_{\beta}(m_1)}.$$
\end{theorem}

In Section \ref{kis0jsmall}, we consider the case that $k=0$ and $j < \lfloor (m_2-1)/2 \rfloor$.  Then Theorem \ref{maintheorem} can be refined as follows.
\begin{theorem}
\label{kis0jsmallthm}
If $k=0$ and $j < \lfloor (m_2-1)/2 \rfloor$, then Theorem \ref{maintheorem} holds with $\epsilon(m_1)$ of the form $$\left(\frac{(2m_2-4j-4)^{\frac{2m_2-4j-5}{2m_2-4j-4}}}{2m_2-4j-5}C^{\frac{1}{2m_2 - 4j - 4}} + o_{m_1}(1)\right)m_1^{\frac{2m_2 - 4j - 5}{2m_2 - 4j - 4}}$$ for a value of $C$ that depends only on $n_1$ and $m_2$.
\end{theorem}
\noindent The value of $C$ will be specified in Section \ref{kis0jsmall}.

We consider $k \geq 1$ and $j < \lfloor (m_2-1)/2 \rfloor$ in Section \ref{othersection}.

\begin{theorem}
\label{otherthm}
If $k \geq 1$ and $j < \lfloor (m_2-1)/2 \rfloor$, then Theorem \ref{maintheorem} holds with $\epsilon(m_1)$ of the form $\Theta(m_1/(\log m_1))$.
\end{theorem}

In Section \ref{kis0jlarge}, we consider the case that $k=0$ and $j = (m_2-2)/2$.  Then the following exact result for large $m_1$ is possible.

\begin{theorem}
\label{kis0jlargethm}
Suppose that $k=0$ and $j = \lfloor (m_2-1)/2 \rfloor$. Theorem \ref{maintheorem} holds with $\epsilon(m_1) = 2^{m_2-1}$ if $m_2$ is even and at least $4$ and $\epsilon(m_1) = 2^{m_2-1}-1$ if $m_2$ is odd or $2$ for sufficiently large $m_1$.
\end{theorem}

In Section \ref{k1jlarge}, we prove the following for $k \geq 1$ and $j = \lfloor (m_2-1)/2 \rfloor$.

\begin{theorem}
\label{k1jlargethm}
If $k \geq 1$ and $j = \lfloor (m_2-1)/2 \rfloor$, then Theorem \ref{maintheorem} holds with $\epsilon(m_1) = 2^{m_2-1}-1$ if $km_1$ is even and $m_2$ is odd, and otherwise $\epsilon(m_1) = 2^{m_2-1} + \lfloor rm_1 \rfloor - rm_1$, for sufficiently large $m_1$.
\end{theorem}

\section{Enriched trees and hypergraphs}
\label{treesection}

\subsection*{Combinatorial species and enriched trees}

We make use of the language of combinatorial species, as presented by Bergeron, Labelle, and Leroux \cite{species}.  A \textit{species} $F$ is, for every finite set $U$, a finite set of objects $F[U]$, called structures, together with, for every bijection $\sigma: U \rightarrow U'$, a function $F[\sigma]: F[U] \rightarrow F[U']$ that satisfies the following two functoriality properties:
\begin{itemize}
\item[1)] for all bijections $\sigma: U \rightarrow U'$ and $\sigma': U' \rightarrow U''$, $F[\sigma' \circ \sigma] = F[\sigma'] \circ F[\sigma]$,
\item[2)] for the identity map $\mbox{Id}_U$, $F[\mbox{Id}_U] = \mbox{Id}_{F[U]}$.
\end{itemize}
The function $F[\sigma]$ is known as transport of species.  Consider the symmetric group $S_U$ that acts on $U$.  Given an $F$-structure $s$, we say that the automorphism group of $s$, $\mbox{aut}(s)$, is the subgroup of those $\sigma \in S_U$ that satisfy $F[\sigma](s) = s$.

Let $\mathfrak{a}$ be the species of asymmetric trees, or trees whose automorphism group is trivial, and let $\mathfrak{a}^\cdot$ be the species of rooted asymmetric trees.  A rooted tree is considered asymmetric if it has no nontrivial root-preserving automorphism; it is possible that the underlying unrooted tree structure is not asymmetric.

Now let $F$ be a species that contains at least one structure over a set of size $1$.  An $F$-\textit{enriched tree} on a set $U$ is a tree on $U$ together with an $F$-structure $s_v$ on the neighbor set $N(v)$ of every vertex $v \in U$.  If $\sigma$ is an automorphism of an $F$-enriched tree $t$, then $\sigma$ is an automorphism of the underlying tree structure of $t$.  Furthermore, if $\sigma_v$ is the restriction of $\sigma$ on $N(v)$, then the transport of species $F[\sigma_v]$ takes $s_v$ to $s_{\sigma(v)}$.  We say that $\mathfrak{a}_F$ is the species of asymmetric $F$-enriched trees.  For example, when $F$ is $\mathcal{E}$, the species of sets, then there is a unique $\mathcal{E}$-structure on every finite set, and $\mathfrak{a}_{\mathcal{E}}$ is simply $\mathfrak{a}$.  The species $\mathfrak{A}$ and $\mathfrak{A}_F$ are, respectively, the species of (not necessarily asymmetric) trees and the species of $F$-enriched trees.

The sum of two species $(F+F')[U]$ is the disjoint union $F[U] + F'[U]$ such that $(F+F')[\sigma](s) = F[\sigma](s)$ if $s \in F[U]$ and $(F+F')[\sigma](s) = F'[\sigma](s)$ if $s \in F'[U]$.  If $a$ is a positive integer, then $aF$ is the sum of $a$ copies of $F$.  We say that $F_i[U]$ is $F[U]$ when $|U| = i$, and otherwise $F_i[U] = \emptyset$, and $|F_i|$ is the number of structures of $F$ over a set with $i$ elements.

Consider the species $F = \sum_{i=1}^\kappa a_i\mathcal{E}_i$ for nonnegative integers $\kappa, a_1, \ldots, a_\kappa$ with $\kappa \geq 1$ and $a_1 \geq 1$.  This is the species that consists of $a_i$ distinct set structures over a set with $i$ elements for $1 \leq i \leq \kappa$ and otherwise no structures.  Then the species $\mathfrak{a}_F$ may be regarded as the species of asymmetric trees in which every vertex has degree at most $\kappa$, and every vertex of degree $i$ is assigned a label from a pool of $a_i$ possible labels.  Such a tree is asymmetric if it has no nontrivial label-preserving automorphism.  Later, we show how estimating the number of elements of $\mathfrak{a}_F$ with a given number of vertices can help determine asymptotic bounds for the distinguishing partitions problem.

A structure of the product species $FF'$ over $U$ is an ordered pair $(f,f')$ for an $F$-structure $f$ over $U_1$ and a $F'$-structure $f'$ over $U_2$ for some partition $U_1 \sqcup U_2$ of $U$.  Transport of species is defined by $(FF')[\sigma](s) = (F[\sigma_1](f),F'[\sigma_2](f'))$, each $\sigma_i$ the restriction of $\sigma$ to $U_i$.

\subsection*{Hypergraphs}

A \textit{hypergraph} $H$ is a triple $(V(H), E(H), I(H))$, with $V(H)$ a finite set of elements called \textit{vertices} and $E(H)$ a finite set of elements called \textit{edges}.  The \textit{incidence relation} $I(H)$ is a subset of $V(H) \times E(H)$.  We will generally treat edges as subsets of $V(H)$.  The \textit{degree} of $v \in V(H)$, denoted $\deg(v)$, is the number of edges incident to $v$.  $H$ is connected if the bipartite graph with vertex sets $V(H)$ and $E(H)$ and edge set $I(H)$ is connected, and $|E(H)| \neq \emptyset$.  If every edge is incident to $n_1$-vertices, then $H$ is $n_1$-\textit{uniform}.

For a hypergraph $H$ with an edge $e$, $\deg_1(H)$ and $\deg_1(e)$ denote the number of vertices of degree $1$ in $H$ and $e$, while $\deg_2^+{H}$ and $\deg_2^+{e}$ are the numbers of vertices of degree at least $2$ in $H$ and $e$.

An automorphism $\sigma$ of $H$ is a permutation of $V(H)$ and $E(H)$ such that $(v, e) \in I(H)$ if and only if $(\sigma(v), \sigma(e)) \in I(H)$ for all vertices $v$ and edges $e$.  We say that $\sigma$ is \textit{trivial} if $\sigma(v) = v$ and $\sigma(e) = e$ for all vertices $v$ and edges $e$, and $H$ is \textit{asymmetric} if the only automorphism of $H$ is trivial.  Thus we allow that hypergraphs may contain multiple edges that are incident to the same vertex set, but such hypergraphs are not asymmetric.

A connected $n_1$-uniform hypergraph $H$ is a \textit{tree} if $E(H)$ can be enumerated $(e_1, \ldots, e_{|E(H)|})$ in such a way that for each $2 \leq i \leq |E(H)|$, we have that $|e_i \cap (e_1 \cup \cdots \cup e_{i-1})| = 1$.  Equivalently, a tree is a connected $n_1$-uniform hypergraph $H$ with $(n_1-1)|E(H)|+1$ vertices.  The \textit{leaves} of a tree $H$ are the edges $e$ that satisfy $\deg_1(e) = n_1-1$.  We say that $l(G)$ is the number of leaves of $G$.

For a tree $G$, define the quantity $$\mu(G) := \sum_{v \in V(G) \atop \deg(v) > 2} (\deg(v)-2).$$  We will later need the following relationship between $l(G)$ and $\mu(G)$.

\begin{lemma}
\label{leafcount}
Let $G$ be a tree with at least $2$ edges.  Then $l(G) \geq \mu(G)+2$.
\end{lemma}

\begin{proof}
Enumerate $E(G) = (e_1, \ldots, e_{|E(G)|})$ such that for $2 \leq i \leq |E(G)|$, $$e_i \cap (e_1 \cup \ldots \cup e_{i-1}) = \{v_i\},$$ and let $G_i$ be the subtree of $G$ with edges $(e_1, \ldots, e_i$).  We prove that the lemma holds for $G_i$ by induction on $i$ for $2 \leq i \leq |E(G)|$, with the $i=2$ case following from $\mu(G_2) = 0$ and $l(G_2) = 2$.  Assume the lemma holds for $G_{i-1}$.

For $3 \leq i \leq |E(G)|$, $\mu(G_i) = \mu(G_{i-1})+1$ if $v_i$ has degree at least $3$ in $G_i$, and otherwise $\mu(G_i) = \mu(G_{i-1})$.  Also, $l(G_i) \geq l(G_{i-1})$ since $e_i$ is a leaf in $G_i$, and at most one leaf of $G_{i-i}$, namely a leaf that contains $v_i$ as a degree $1$ vertex, is not a leaf in $G_i$.  Furthermore, whenever $\mu(H_i) = \mu(G_{i-1})+1$, $v_i$ has degree at least $2$ in $G_{i-1}$ and thus $l(G_i) = l(G_{i-1})+1$.  The lemma follows for $G_i$ by $l(G_i) - l(G_{i-1}) \geq \mu(G_i) - \mu(G_{i-1})$.
\end{proof}

\section{Distinguishing partitions and asymmetric hypergraphs}
\label{corrsection}

We demonstrate a bijection between distinguishing partitions of complete multipartite graphs and asymmetric hypergraphs with prescribed edge sizes.  Using this bijection, we establish the existence or nonexistence of distinguishing partitions by demonstrating the existence or nonexistence of certain asymmetric hypergraphs.  The argument is nearly identical to that given by Ellingham and Schroeder \cite{distpart}.

With $X = K_{n_1, \ldots, n_m}$, let $P$ be a partition of $X$ with parts $P_1, \ldots, P_t$, and we say that $P$ is a \textit{regular partition} of $X$ if $|X_i \cap P_{i'}| \leq 1$  for all $i$ and $i'$.  It is a necessary but not sufficient condition for $P$ to be distinguishing that $P$ is regular.

\begin{definition}
For every regular partition $P$ of $X$ with parts $P_1, \ldots, P_t$, we associate a hypergraph $\tau(P)$ as follows: $V(\tau(P)) = \{P_i, 1 \leq i \leq t\}$, $E(\tau(P)) =$ $\{X_i: 1 \leq i \leq m\},$ and $X_i$ and $P_{i'}$ are incident if $|X_i \cap P_{i'}| = 1$.
\end{definition}

Note that $\tau(P)$ is a hypergraph with $m$ edges with sizes $n_1, \ldots, n_m$ since $X_i$ intersects exactly $n_i$ parts of $P$.

We say that the automorphism group $\mbox{aut}(P)$ is the subgroup of $\mbox{aut}(X)$ consisting of those elements that fix $P$.  The following relationship holds.

\begin{lemma}
If $P$ is a regular partition of $X$, then $\mbox{\upshape aut}(P)$ is isomorphic to $\mbox{\upshape aut}(\tau(P))$.
\end{lemma}
\begin{proof}
Let $\tilde{\tau}: \mbox{aut}(P) \rightarrow \mbox{aut}(\tau(P)$ be the group homomorphism induced by $\tau$.  Say that $P_1, \ldots, P_t$ are the parts of $P$.

An automorphism $\sigma \in \mbox{aut}(P)$ induces automorphisms $\sigma_P$ and $\sigma_X$ on the sets $\{P_i\}$ and $\{X_{i'}\}$ respectively.  Then $\sigma$ is uniquely determined by $\sigma_P$ and $\sigma_X$, and in particular $\sigma$ is trivial if and only if $\sigma_P$ and $\sigma_X$ are both trivial.  Thus $\tilde{\tau}$ is injective.

Now let $\sigma' \in \mbox{aut}(\tau(P))$.  Then $\sigma'$ is uniquely determined by incidence-preserving permutations of $\{P_i\}$ and $\{X_{i'}\}$.  Let $\sigma$ be the permutation of $X$ such that if $x \in X$ is the unique vertex contained in $X_i \cap P_{i'}$, then $\sigma(x)$ is the unique vertex contained in $\sigma'(X_i) \cap \sigma'(P_{i'})$.  It is readily checked that in fact $\sigma \in \mbox{aut}(P)$, and thus $\tilde{\tau}$ is surjective.
\end{proof}

\begin{corollary}
\label{distparthypergraph}
There exists a distinguishing partition of $K_{n_1, \ldots, n_m}$ if and only if there exists an asymmetric hypergraph with $m$ edges of sizes $n_1, \ldots, n_m$.
\end{corollary}

It will be convenient to associate another hypergraph with a regular partition $P$ of $K_{m_1(n_1),m_2(n_2)}$.  Let $\tau'(P)$ be a vertex-labeled hypergraph that contains exactly the vertices and the $n_1$-edges of $\tau(P)$.  Say that the $n_2$-edges of $\tau(P)$ are $X_1, \ldots, X_{m_2}$.  Then the vertex label set of $\tau'(P)$ is $2^{[m_2]}$, and a vertex $v$ in $\tau'(P)$ is labelled with a set $S \subseteq [m_2]$ if $v \in X_i$ exactly when $i \in S$.  Then $\tau'(P)$ is just a different way of encoding $\tau(P)$.

The \textit{weight} $w(v)$ of a vertex $v \in \tau'(P)$ is the cardinality of its label.  The \textit{weight} $w(S)$ of a set of vertices $S$ is the sum of the weights of the vertices in $S$.  The \textit{weight} $w(e)$ or $w(G)$ of an $n_1$-edge $e$ or connected component $G \subset \tau'(P)$ is the weight of the vertex set of $e$ or $G$.  The \textit{value} of $G$ is $w(G) - rm_2|E(G)|$.  Value may be positive or negative.  We have that $n_2 = w(\tau'(P))/m_2$, and thus our strategy in proving the main results is to find an asymmetric labelled hypergraph with $m_1$ $n_1$-edges and maximal weight.  Though weight and value encode the same information, value is useful in that it gives a clear comparision of the weight of a component of $\tau'(P)$ to its asymptotic limit.

\section{Key Lemmas}
\label{generallemmas}

We now present a series of lemmas that provide upper bounds on the weights and values of certain types of components.

\begin{lemma}
\label{vertexcount}
Let $G$ be a connected $n_1$-uniform hypergraph with $n_1|E(G)|/2 + p$ vertices.  Then $G$ contains $2p+\mu(G)$ vertices of degree $1$.  Equivalently, each edge has, on average, $(2p+\mu(G))/|E(G)|$ degree $1$ vertices.
\end{lemma}
\begin{proof}
There are $n_1|E(G)|$ pairs of the form $(v,e)$, where $v$ is a vertex, $e$ an edge, and $v \in e$.  The number of such pairs $(v,e)$ is also $$\deg_1(G) + 2(n_1|E(G)|/2 + p - \deg_1(G)) + \mu(G),$$ or $n_1|E(G)| + 2p - \deg_1(G) + \mu(G)$.  Thus $\deg_1(G) = 2p+\mu(G)$.
\end{proof}

\begin{lemma}
\label{weightint}
Suppose that $\tau(P)$ is asymmetric, and let $S$ be the set of degree $1$ vertices in an edge of $\tau'(P)$.  Then $|S| \leq 2^{m_2}$.  Define nonnegative values $j'$ and $k'$ such that $|S| = {m_2 \choose 0} + \cdots + {m_2 \choose j'} + k'$ with either $0 \leq k' < {m_2 \choose j'+1}$ or $k'=0$ and $j' = m_2$.  Then
$$w(S) \leq \sum_{i=0}^{j'} (m_2-i){m_2 \choose i} + k'(m_2 - j' -1).$$
\end{lemma}
\begin{proof}
For all $v_1, v_2 \in S$, there is an automorphism of the underlying unlabelled hypergraph of $\tau'(P)$ that switches $v_1$ and $v_2$ and fixes all other vertices.  Thus all vertices in $S$ must have different labels in $\tau'(P)$, which implies that $|S| \leq 2^{m_2}$.  The lemma follows from the fact that $S$ contains at most ${m_2 \choose i}$ vertices with a label of cardinality $m_2-i$.
\end{proof}

Let $w_{|S|}$ denote the upper bound on $w(S)$ in Lemma \ref{weightint}.  Now suppose that $\tau'(P)$ is asymmetric and $G$ is a component of $\tau'(P)$.  A \textit{defect} in $G$ is one of the following.  A \textit{defective vertex} is a vertex $v$ with degree at least $2$ and weight less than $m_2$, counted with multiplicity $d(v) = m_2-w(v)$.  A \textit{defective} edge is an edge $e$ with set $S$ of degree $1$ vertices with collective weight less than $w_{|S|}$, counted with multiplicity $d(e) = w_{|S|}-w(S)$.  The number of defects in $G$ is denoted by $d(G)$.

\begin{lemma}
\label{weightcount1}
Let $G$ be a connected component of $\tau'(P)$.  If $\tau(P)$ is asymmetric, then $$w(G) \leq m_2\deg_2^+(G) + \sum_{e \in E(G)} w_{\deg_1(e)} - d(G).$$
\end{lemma}

\begin{proof}
Every vertex has weight at most $m_2$, and a defective vertex $v$ with degree at least $2$ has weight $m_2-d(v)$.  The set of degree $1$ vertices in an edge $e$ has weight $w_{\deg_1(e)}$ if $e$ is not defective, and otherwise weight $w_{\deg_1(e)} - d(e)$.  The lemma follows by adding over all vertices.
\end{proof}

If $G$ contains $n_1|E(G)|/2 + p$ vertices, then write $$2p+\mu(G) = b \left\lfloor \frac{2p+\mu(G)}{|E(G)|} \right\rfloor + b' \left\lceil\frac{2p+\mu(G)}{|E(G)|} \right\rceil$$ with nonnegative $b+b' = |E(G)|$.  The following lemma states that the weight of $G$ is maximized when all edges have about the same number of degree $1$ vertices.

\begin{lemma}
\label{weightbound}
With all quantities as above, $$w(G) \leq m_2n_1|E(G)|/2 - m_2p - m_2\mu(G) + bw_{\left\lfloor \frac{2p+\mu(G)}{|E(G)|} \right\rfloor} + b'w_{\left\lceil \frac{2p+\mu(G)}{|E(G)|} \right\rceil} - d(G).$$
\end{lemma}
\begin{proof}
By Lemma \ref{vertexcount}, $G$ has $n_1|E(G)|/2 - p - \mu(G)$ vertices of degree at least $2$.  Then by Lemma \ref{weightcount1}, $$w(G) \leq m_2n_1|E(G)|/2 - m_2p - m_2\mu(G) + \sum_{e \in E(G)} w_{\deg_1(e)} - d(G).$$ The expression $w_y$ is concave in $y$, meaning that for all $y$, $w_y - w_{y-1} \geq w_{y+1} - w_y$.  Thus, given a set of values $\{y_i\}$ such that $\sum_i y_i = 2p+\mu(G)$, $\sum_{i=1}^t w_{y_i}$ is maximal when $b$ of the $y_i$ are equal to $\left\lfloor \frac{2p+\mu(G)}{|E(G)|} \right\rfloor$ and $b'$ of the $y_i$ are $\left\lceil\frac{2p+\mu(G)}{|E(G)|} \right\rceil$.  The lemma follows.
\end{proof}

We now look to maximize the weight of $G$ by considering the total number of vertices of a given weight.  In particular, $G$ contains at most $|E(G)|{m_2 \choose i}$ degree $1$ vertices with weight $m_2-i$, since each each edge contains at most ${m_2 \choose i}$ such vertices.  Choose values $j^*$ and $k^*$ such that $$2p+\mu(G) = |E(G)|{m_2 \choose 0} + \cdots + |E(G)|{m_2 \choose j^*} + k^*$$ with $j^* \geq -1$ and $0 \leq k^* < |E(G)|{m_2 \choose j^*+1}$.

\begin{lemma}
\label{weightboundglobal}
With all quantities as above, $w(G) \leq$ $$m_2\left(\frac{n_1|E(G)|}{2} - p - \mu(G)\right) + \sum_{i=0}^{j^*}|E(G)|(m_2-i){m_2 \choose i} + (m_2-j^*-1)k^* - d(G).$$
\end{lemma}
\begin{proof}
Let $G'$ be a (not necessarily asymmetric) hypergraph constructed from $G$ by giving every vertex of degree at least $2$ the label $[m_2]$ and assigning a label to every degree $1$ vertex such that all edges of $G'$ are nondefective and have distinct labels among the degree $1$ vertices.  Then $G'$ has $(n_1|E(G)|/2 - p - \mu(G))$ vertices of degree at least $2$, each of which has weight $m_2$, and at most $|E(G)|{m_2 \choose i}$ degree $1$ vertices of weight $m_2-i$ for $0 \leq i \leq j^*$.  The result follows by $w(G') = w(G)+d(G)$.
\end{proof}

We consider the upper bound of Lemma \ref{weightboundglobal} to be a function $w_{\max}(p,\mu(G),d)$, with the quantities $m_2$ and $|E(G)|$ considered to be fixed.  We note that $$w_{\max}(p,\mu(G),d) > w_{\max}(p,\mu(G),d+1),$$ while $$w_{\max}(p,\mu(G),d) \geq w_{\max}(p,\mu(G)+1,d),$$ with equality exactly when $j^*=-1$.

Now we consider $\mu(G) = d = 0$, and the upper bound of Lemma \ref{weightboundglobal} is a function $w_{\max}(p)$.  The effect of replacing $p$ by $p+1$ is equivalent, numerically, to replacing a vertex with weight $m_2$ by two vertices, one of weight $m_2-j^*-1$ and the other of weight either $m_2-j^*-1$ or $m_2-j^*-2$.  Thus the function $w_{\max}(p)$ is weakly unimodal in $p$ and achieves a maximum when $j^* = \lfloor (m_2 - 1)/2 \rfloor$ and $k^*=0$ or $1$.  Then $2p = \sum_{i=0}^{\lfloor (m_2 - 1)/2 \rfloor}|E(G)|{m_2 \choose i}+(0$ or $1)$, and we have by Lemma \ref{weightboundglobal} that 
\begin{equation}
\label{wGbound}
w(G) \leq |E(G)|\left(\frac{m_2n_1}{2} - \frac{m_2}{2} \sum_{i=0}^{\lfloor \frac{m_2 - 1}{2} \rfloor}{m_2 \choose i} + \sum_{i=0}^{\lfloor \frac{m_2 - 1}{2} \rfloor}(m_2-i){m_2 \choose i} \right).
\end{equation}
  Thus the following holds.

\begin{corollary}
\label{valuecor}
Let all quantities be as above. 
\begin{enumerate}
\item If $j = \lfloor (m_2 - 1)/2 \rfloor$, then $w(G) \leq m_2r|E(G)|$.
\item If $j < \lfloor (m_2 - 1)/2 \rfloor$ and $p \leq |E(G)|(\frac{n_1}{2}-1)$, then $w(G) \leq m_2r|E(G)|$.
\item $G$ has positive value only if $j < \lfloor (m_2 - 1)/2 \rfloor$ and $G$ is a tree.  Then if $G$ has $d$ defects, $v(G) \leq m_2 - 2j - d$.
\end{enumerate}
\end{corollary}

\begin{proof}
Part 1 follows by Equation (\ref{wGbound}) and the definition of $r$.  Part 2 follows from the monotonicity of $w_{\max}$ and the definition of $r$.  Part 3 is a consequence of Parts 1 and 2.
\end{proof}

We now focus on the particular case that $G$ is a tree and $k=0$.  Suppose that $G$ has $l$ leaves, and by Lemma \ref{leafcount}, $l \geq \mu(G)+2$.  Then $j^* = j$ and $k^* = 2+\mu(G)$, and $\sum_{e \in E(G)} w_{\deg_1(e)} \leq$ $$\sum_{i=0}^{j}|E(G)|(m_2-i){m_2 \choose i}+ (\mu(G)+2)(m_2-j-1).$$ This bound can be attained if $G$ contains $|E(G)|{m_2 \choose i}$ degree $1$ vertices of weight $m_2-i$ for $0 \leq i \leq j$ and $\mu(G)+2$ degree $1$ vertices of weight $m_2-j-1$.  However, every leaf of $G$ contains a vertex of weight at most $m_2-j-1$, and thus in fact $\sum_{e \in E(G)} w_{\deg_1(e)} \leq$ $$\sum_{i=0}^{j}|E(G)|(m_2-i){m_2 \choose i}+ (\mu(G)+2)(m_2-j-1) - (l-\mu(G)-2).$$  Since $G$ has $|E(G)| - 1 - \mu(G)$ vertices of degree $2$ or more, $$w(G) \leq m_2|E(G)| + \sum_{i=0}^{j}|E(G)|(m_2-i){m_2 \choose i}- \mu(G)j+m_2-2j-l.$$

Finally, if we allow that $G$ might have $d$ defects, then we conclude the following.

\begin{lemma}
\label{treevalue}
With all quantities as above, $v(G) \leq m_2 - 2j - l - j\mu(G) - d$.
\end{lemma}

Now we consider all of $\tau'(P)$.  Let $\mbox{Comp}(P)$ be the set of connected components of $\tau'(P)$.

\begin{lemma}
\label{valuetotal}
If $\tau(P)$ is asymmetric, $v(\tau'(P)) \leq \sum_{G \in \mbox{\upshape{Comp}}(P)}v(G) + m_22^{m_2-1}$.
\end{lemma}
\begin{proof}
We calculate that $v(\tau'(P))$ is $\sum_{G \in \mbox{\upshape{Comp}}(P)}v(G)$ plus the sum of the weights of all vertices not contained in any $n_1$-edge.  Since $\tau(P)$ is asymmetric, every vertex not contained in an $n_1$-edge must have a different label, and there is at most one vertex with every label $S \subset 2^{[m_2]}$.  The lemma follows.
\end{proof}

\section{An extremal construction}
\label{modelconst}

In this section, we give a general method of constructing a distinguishing partition $P$ of $K_{m_1(n_1),m_2(n_2)}$.  We then show that $n_2$ is maximal to within an additive constant, given the other parameters.  We do so by describing the vertex-labeled hypergraph $\tau'(P)$.  For the remainder of this section, we assume that $j < \lfloor (m_2-1)/2 \rfloor$; the case that $j = \lfloor (m_2-1)/2 \rfloor$ is treated seperately.

Let $G$ be a component of $\tau'(P)$.  Define the \textit{value} $v(e)$ of an edge $e \in E(G)$ to be $v(G)/|E(G)|$.  Let $\xi$ be the map that adds $1$ mod $m_2$ to every element in the label of every vertex of a $2^{[m_2]}$-vertex lableled hypergraph.  Note that $v(\xi(G)) = v(G)$.  Say that vertex labeled hypergraphs $G,G'$ are equivalent under $\sim_{\xi}$ if $G = \xi^i(G')$ for some $i$.

Let $\mathcal{T}^* = \mathcal{T}^*_{n_1,m_2} = (\mathcal{T}_1, \mathcal{T}_2, \ldots)$ be an ordered list of equivalence classes under $\sim_{\xi}$ of positive weight asymmetric hypergraphs such that an edge in an element of $\mathcal{T}_i$ has value at least as great as an edge in an element of $\mathcal{T}_{i+1}$ for all $i$.  By Lemma \ref{treevalue}, an edge $e$ may have value $\delta > 0$ only if $e$ is contained in a tree with at most $m_2/\delta$ edges.  Thus $\mathcal{T}^*$ enumerates all hypergraphs with positive value, and only classes of trees are in $\mathcal{T}^*$.

A \textit{symmetry breaking ring} $R$ is a $2^{[m_2]}$-vertex labeled hypergraph on at least $\min_R = \min_R(n_1,m_2)$ edges defined as follows.  If $m_2 = 1$, then $j = -1$ and $n_1 = 2$, and we set $\min_R = 6$.  Let $R$ be a cycle that contains consecutive vertices $v_1,v_2,v_3,v_4$.  Then all vertices of $R$ have weight $1$ except for $v_1,v_2,v_4$, which all have weight $0$.

If $m_2 > 1$, we set $\min_R = \max(2m_2,n_1+1)$.  Let $\mbox{quot}_R$ be the maximum multiple of $m_2$ up to $|E(R)|$.  Let $v_0, \ldots, v_{\mbox{quot}_R-1}$ be vertices and $e_0, \ldots, e_{\mbox{quot}_R-1}$ be edges such that $e_i$ contains only degree $1$ vertices except for $v_i, v_{i+1}$, subscripts mod $\mbox{quot}_R$.  The set of vertex labels of the degree $1$ vertices of $e_0$ includes all possible labels of size at least $m_2-j$, and all others are of size $m_2-j-1$.  To determine the labels of the degree $1$ vertices of $e_i$, add $i$ mod $m_2$ to every element in the labels of the degree $1$ vertices of $e_0$.  Assign $v_0$ the label $\emptyset$, $v_i$ the label $[m_2]-{i}$ for $1 \leq i \leq m_2$, and $v_i$ the label $[m_2]$ for all other $i$.  Finally, $R$ contains edges of the form ${v_i, v_{i+1}, \ldots, v_{i+n_1-1}}$ for $1 \leq i \leq |E(R)| - \mbox{quot}_R$.

We need some key facts on symmetry breaking rings.

\begin{lemma}
Let $R$ be a symmetry breaking ring that is a component of $\tau'(P)$.  Then no automorphism of $P$ induces a nontrivial automorphism of $R$.
\end{lemma}
\begin{proof}
The lemma is readily verified when $m_2 = 1$, and so we assume that $m_2 > 1$.  Let $\sigma$ be an automorphism of $\tau(P)$ that induces an automorphism of $R$.  Since $v_0$ is the only vertex of $R$ of weight $0$, it is a fixed point.  Since $v_1$ is the only vertex of $R$ that is in a common edge with $v_0$, has degree $2$ in $\tau'(P)$, and weight not equal to $m_2$, $v_1$ is also a fixed point.  Thus all $v_i$ are fixed, which implies that $\sigma$ fixes the $n_2$-edges of $\tau(P)$.  Since all $v_i$ are fixed, $\sigma$ thus also fixes all $n_1$-edges of $R$.  Finally, since all degree $1$ vertices in a given edge have different labels, they must be fixed points as well.
\end{proof}

The following is readily observed from the construction of symmetry breaking rings.

\begin{lemma}
Let $R$ be a symmetry breaking ring, and let $1 \leq i < i' \leq m_2$.  Then the number of vertices of $R$ whose label contains $i$ is equal to the number of vertices of $R$ whose label contains $i'$.
\end{lemma}

\begin{lemma}
There exists a value $\omega = \omega_{n_1,m_2}$, which depends only on $n_1$ and $m_2$, such that a symmetry breaking ring $R$ has value at least $\omega$.
\end{lemma}
\begin{proof}
If $m_2 = 1$, then $v(R) = -3$.  If $m_2 > 1$, then the total weight of the degree $1$ vertices in each edge with degree $1$ vertices is $m_2(r-1)$.  All other vertices have weight $m_2$ except for $m_2$ vertices of weight $m_2-1$ and one of weight $0$.  It follows that $w(R) = \mbox{quot}_R m_2r - 2m_2$ and  $v(R) = -(|E(R)| - \mbox{quot}_R)m_2r - 2m_2 > -m_2^2r - 2m_2$.
\end{proof}

We now come to our construction of $P$.  Choose $\zeta$ to be the maximum value such that the total number of edges in all trees of $\mathcal{T}_1 \cup \cdots \cup \mathcal{T}_\zeta$ is at most $m_1 - \min_R$.  Let $\Delta_{m_1(n_1),m_2}$ be the union of all trees in $\mathcal{T}_1 \cup \cdots \cup \mathcal{T}_\zeta$, together with a symmetry breaking ring so that $\Delta_{m_1(n_1),m_2}$ has $m_1$ edges, and a degree $0$ vertex of every label except $\emptyset$.

\begin{lemma}
$\Delta_{m_1(n_1),m_2}$ is in fact $\tau'(P)$ for a distinguishing partition $P$ for an appropriate value of $n_2$.
\end{lemma}
\begin{proof}
By construction, $\Delta_{m_1(n_1),m_2}$ contains the same number of vertices whose label contains $i$ as the number of vertices whose label contains $i'$ for all $1 \leq i < i' \leq m_2$.  Thus $\Delta_{m_1(n_1),m_2}$ is $\tau'(P)$ for some partition $P$ of $K_{m_1(n_1),m_2(n_2)}$ and for some $n_2$.  Next, we apply Corollary \ref{distparthypergraph} and show that $P$ is distinguishing by showing that $\tau(P)$ is asymmetric.  Let $\sigma$ be an automorphism of $\tau(P)$.  Since $\Delta_{m_1(n_1),m_2}$ contains exactly one symmetry breaking ring, all $n_2$-edges of $\tau(P)$ are fixed  under $\sigma$.  Since all components of $\Delta_{m_1(n_1),m_2}$ are asymmetric and no two are isomorphic to each other, all $n_1$-edges and vertices of $\tau(P)$ are fixed as well.
\end{proof}

Next, we prove that $\Delta_{m_1(n_1),m_2}$ is a nearly optimal construction.

\begin{lemma}
\label{nearopt}
Let $P$ be a distinguishing partition of $K_{m_1(n_1),m_2(n_2)}$ such that $\tau'(P) = \Delta_{m_1(n_1),m_2}$, and let $G'$ be an asymmetric vertex-labeleld hypergraph.  Then $w(G') \leq w(\Delta_{m_1(n_1),m_2}) + \mbox{\upshape{Error}}_{n_1,m_2}$ for some value $\mbox{\upshape{Error}}_{n_1,m_2}$ that depends only on $n_1$ and $m_2$.  In particular, if $P'$ is a distinguishing partition of $K_{m_1(n_1),m_2(n_2')}$, then $n_2' \leq n_2 + \mbox{\upshape{Error}}_{n_1,m_2}/m_2$.
\end{lemma}
\begin{proof}
First we determine an upper bound on $w(G')$ in terms of the structure $\mathcal{T}^*$, and then we compare that to $w(\Delta_{m_1(n_1),m_2})$.  Let $v^+$ be the sum of the weights of all vertices of $G'$ that are not contained in edges; since they must all have different labels, $v^+ \leq m_2 2^{m_2-1}$.  Then, summing over all components $G$ of $G'$ that contain $n_1$-edges, $n_2' \leq \frac{1}{m_2}\sum_{G} w(G) + 2^{m_2-1} = rm_1 + \frac{1}{m_2}\sum_{e \in E(G')}v(e) + 2^{m_2-1}$.

Suppose that the set of edge values of hypergraphs in $\mathcal{T}^*$ are $\{v_1, v_2, \ldots\}$ with $v_1 > v_2 > \cdots$, and suppose that there are $\mbox{nv}_i$ total edges in all hypergraphs of $\mathcal{T}^*$ with value $v_i$.  Let $\rho$ be the largest value such that $\tau'(P)$ contains $\mbox{nv}_{\rho}$ edges of value $v_{\rho}$.  Then $$w(G') \leq rm_1m_2 + \sum_{i=1}^{\rho}\mbox{nv}_iv_i + \left( m_1-\sum_{i=1}^{\rho}\mbox{nv}_i\right) v_{\rho+1} + m_22^{m_2-1}.$$

Now we consider $\tau'(P)$ with symmetry breaking ring $R$.  Choose $\zeta$ so that $\mathcal{T}_{\zeta}$ is the last equivalence class of trees that are components of $\tau'(P)$; edges in $\mathcal{T}_{\zeta+1}$ have value $v_{\rho+1}$, and thus by Lemma \ref{treevalue}, each tree in $\mathcal{T}_{\rho+1}$ has at most $m_2/(v_{\rho+1})$ edges.  By construction, $R$ has at most $m_2^2/(v_{\rho+1}) + \min_R$ edges, each of which has value at least $\omega/|E(R)|$.  Furthermore, $\tau'(P)$ contains $\mbox{nv}_i$ edges of value $v_i$  for $1 \leq i \leq \rho$, and all edges besides these and edges in $R$ have value $v_{\rho+1}$.  Thus $w(\Delta_{m_1(n_1),m_2}) \geq $

$$rm_1m_2 + \sum_{i=1}^{\rho}\mbox{nv}_iv_i + \left( m_1-\sum_{i=1}^{\rho}\mbox{nv}_i\right) v_{\rho+1} - |E(R)|\left(v_{\rho+1} - \frac{\omega}{|E(R)|}\right) + m_22^{m_2-1}.$$

Thus $w(G') - w(\Delta_{m_1(n_1),m_2}) \leq |E(R)|(v_{\rho+1} - \omega/|E(R)|) \leq m_2^2 + \min_Rv_{\rho+1} - \omega$.  This proves the lemma.
\end{proof}

In the subsequent sections, we prove upper bounds on $n_2$, in terms of the other variables, by evaluating the weights of $\Delta_{m_1(n_1),m_2}$.  For every $m_1,n_1,m_2$, choose $n_2'(m_1,n_1,m_2)$ maximally so that $K_{m_1(n_1),m_2(n_2')}$ has a distinguishing partition $P'_{m_1(n_1),m_2}$.

\begin{lemma} $\lim_{m_1 \rightarrow \infty}\frac{n_2'(m_1,n_1,m_2)}{n_2'(m_1+1,n_1,m_2)} = 1$.
\label{ratio1}
\end{lemma}
\begin{proof}
It follows by construction of $\Delta_{m_1(n_1),m_2}$ and Lemma \ref{nearopt}.
\end{proof}

\begin{lemma}
If $m_1$ is large and $n_2$ is small, then $K_{m_1(n_1),m_2(n_2)}$ has a distinguishing partition.
\end{lemma}
\begin{proof}
If $n_2$ is small relative to $n_1$ and $m_2$, then an asymmetric hypergraph with $m_i$ edges of size $n_i$ for $i=1,2$ may be constructed as follows.  First take an asymmetric hypergraph with $m_1 $ $n_1$-edges, which exists by the main result of \cite{distpart}.  Then add $m_2$ $n_2$-edges on the same vertex set.  The result is asymmetric.
\end{proof}

Assume that $n_2$ is large.  Choose $m^*$ maximally so that $n_2' = n_2'(m^*,n_1,m_2) \geq n_2$.  By Lemma \ref{ratio1}, $n_2'/n_2$ is close to $1$.  We need to show that $K_{m^*(n_1),m_2(n_2)}$ has a distinguishing partition.  Our method is to show that there exists distinguishing partition $P'$ of $K_{m^*(n_1),m_2(n_2')}$ and a subset $S$ of weight $m_2$ vertices on $\tau'(P')$, with $|S| = n_2'-n_2$, such that the hypergraph that results by changing all of the labels of vertices of $S$ from $[m_2]$ to $\emptyset$ is asymmetric.

\begin{lemma}
With all quantities as above, $V(\tau'(P'))$ has a subset $S$ of weight $m_2$ vertices of size $n_2'-n_2$ such that the hypergraph $G'$ that results from changing all labels of vertices of $S$ from $[m_2]$ to $\emptyset$ is $\tau'(P^*)$ for a distinguishing partition $P^*$ of $K_{m^*(n_1),m_2(n_2)}$.
\end{lemma}

\begin{proof}
Certainly $P^*$ is a partition of $K_{m^*(n_1),m_2(n_2)}$.  It suffices to show that $S$ may be chosen so that $\tau(P^*)$ is asymmetric.  If $\tau'(P')$ contains $\Theta(n_2')$ vertices of degree at least $2$ of weight $m_2$ and none of weight $0$, then any set $S$ of size $n_2'-n_2$ of weight $m_2$ and degree at least $2$ vertices satisfies the desired property.  This condition is seen directly in all cases that $j = \lfloor (m_2-1)/2 \rfloor$, as $P'$ is constructed directly in subsequent sections.

It must be that $\tau'(P')$ has few components with weight $0$ vertices, and no components with many weight $0$ vertices.  Otherwise, if $n_1=2$, we could replace all those components and a largest defect-free component of $\tau'(P')$, if there is one, by a single defect-free asymmetric tree, which would increase the weight, a contradiction to Lemma \ref{nearopt}.  Otherwise, we could replace all those components and a largest symmetry breaking ring of $\tau'(P')$, if there is one, with a single symmetry breaking ring.  By Lemma \ref{treevalue} if $n_1 = 3$, and otherwise by Corollary \ref{valuecor}, this would increase the weight, also a contradiction to Lemma \ref{nearopt}.

It is shown in subsequent sections that for all sufficiently large $t$, there is an element of $\mathcal{T}^*$ with $t$ edges.  Thus there are $\Omega(\sqrt{m^*})$ elements of $\mathcal{T}^*$ of size up to $2\sqrt{m_1}$ that are not components of $\tau'(P')$.  It must be that all but $o(m^*)$ edges of $\tau'(P')$ are contained in positive weight components; otherwise, all components with nonpositive weight could be removed and replaced by $\Omega(\sqrt{m^*})$ components of positive weight, a contradiction to Lemma \ref{nearopt}.

Let $T_1, \ldots, T_a$ be the components of $\tau'(P')$ with weight $0$ vertices.  Let $T$ be a positive weight component with $t$ vertices.  Then $T$ has $\Theta(t)$ vertices of degree at least $2$, all but at most $m_2$ of which have weight $m_2$, and thus $\tau'(P')$ has $\Theta(m^*)$ vertices in positive weight components with weight $m_2$.  Let $S$ be a subset of size $n_2'-n_2$, chosen uniformly at random, of the weight $m_2$ vertices that are contained in the larger half of positive weight components.  Let $G'$ be the hypergraph that results from changing the labels of all vertices of $S$ in $\tau'(P')$ from $[m_2]$ to $\emptyset$.

Every component $T$ of $G'$ that contains a vertex of $S$ is asymmetric, since it was constructed by dividing the set of vertices labelled $[m_2]$ into vertices labelled $[m_2]$ and $\emptyset$, and it had no vertex labelled $\emptyset$ previously.  $T$ is not isomorphic to another component $T'$ that contains a vertex of $S$ since the hypergraph that results from changing all vertices of $T$ of label $\emptyset$ to $[m_2]$ is nonisomorphic to the hypergraph that results from changing all vertices of $T'$ of label $\emptyset$ to $[m_2]$.  To conclude, we need to show that with high probability, $T_i$ is not isomorphic to any component of $G'$ for each $1 \leq i \leq a$, since $a$ is small.

If $T_i$ is isomorphic to $T$, a component of $G'$ that contains a vertex of $S$, then it must be that the hypergraphs $T_i^*$ and $T^*$, which result from converting all vertices of $T_i$ and $T$ of label $\emptyset$ to $[m_2]$, are isomorphic.  Thus $T_i^*$ is isomorphic to a component $T^*$ of $\tau'(P')$.  Since $T^*$ is asymmetric, for all subsets $S'$ of weight $m_2$ vertices of $T^*$, the hypergraphs that result from converting all vertices of $S'$ from label $[m_2]$ to $\emptyset$ are nonisomorphic.  Since $T^*$ is large, the probability that $T_i$ is isomorphic to some other component of $G'$ is small.
\end{proof}

\begin{lemma}
If $K_{m^*(n_1),m_2(n_2)}$ has a distinguishing partition and $m_1 > m^*$, then $K_{m_1(n_1),m_2(n_2)}$ has a distinguishing partition.
\end{lemma}
\begin{proof}
It suffices to prove the lemma for $m_1 = m^*+1$.  Let $P$ be a distinguishing partition of $K_{m^*(n_1),m_2(n_2)}$.  Define a \textit{tail T of} $\tau(P)$ to be a sequence of vertices $v_1, \ldots, v_{n_1+t-1}$ and edges $\{v_i, \ldots, v_{n_1+i-1}\}$ for $1 \leq i \leq t$, such that $v_{n_1}, \ldots, v_{n_1+t-1}$ are contained in no edges outside of $T$.  Assume that $T$ is chosen so that $t$ is maximal.  Then add an vertex $v_{n_1+t}$ and an edge $\{v_{t+1},\ldots,v_{n_1+t}\}$ to $\tau(P)$ to create a hypergraph $G'$.

We show that $G'$ is asymmetric.  By construction, $v_{n_1+t}$ is the only degree $1$ vertex contained in a maximum tail of $G'$, and thus it is a fixed point.  Then the edge $\{v_{t+1},\ldots,v_{n_1+t}\}$ is fixed, since it is the only edge to contain $v_{n_1+t}$.  Thus all other vertices and edges are fixed as well, since $\tau(P)$ is asymmetric.  It follows that $K_{m_1(n_1),m_2(n_2)}$ has a distinguishing partition.
\end{proof}

We summarize the preceding lemmas as follows.

\begin{corollary}
If $m_1$ is large relative to $n_1$ and $m_2$, let $n_2'$ be the largest value such that $K_{m_1(n_1),m_2(n_2')}$ has a distinguishing partition.  Then $K_{m_1(n_1),m_2(n_2)}$ has a distinguishing partition if $n_2 \leq n_2'$.
\end{corollary}

\section{$n_1 = 2$}
\label{n1is2}

In this section we consider the case that $n_1 = 2$.  A labelled connected $2$-uniform hypergraph has positive value only if it is a tree with fewer than $m_2$ defects.  Furthermore, all trees without defects have positive value.

Our bounds and construction requires an estimate on the number of asymmetric ordinary trees, which is provided by the twenty-step algorithm of Harary, Robinson, and Schwenk.

\begin{lemma}
\label{ordinarytreecount}
There exists constants $\alpha > 0$ and $\beta > 1$ such that the number of asymmetric trees on $i$ edges is $(1+o_i(1))\alpha \beta^i i^{-5/2}$.  Furthermore, there exists $\alpha' > 0$ such that the number of asymmetric rooted trees on $i$ edges is $(1+o_i(1))\alpha' \beta^i i^{-3/2}$.
\end{lemma}

\proofof{Theorem \ref{n1is2thm}}
Recall that $$z = \left\lfloor {\log_{\beta}\left(\frac{m_1(\beta-1)}{\alpha \beta}\left({\log_\beta m_1}\right)^{3/2}\right)} \right\rfloor.$$

Summing the result of Lemma \ref{ordinarytreecount} from $1$ to $z$, the number of nonisomorphic asymmetric trees with at most $z$ edges is $(\alpha \frac{\beta}{\beta-1} + o(1)) \beta^z z^{-5/2}$, and they collectively have $(\alpha \frac{\beta}{\beta-1} + o(1)) \beta^z z^{-3/2} \leq (1+o(1))m_1$ edges.  Similarly, the collective number of edges of nonisomorphic asymmetric trees with at most $z+1$ edges is at least $(1+o(1))m_1$, and with at most $z+2$ edges exceeds $m_1$.  Each edge of a defect-free tree on $z+2$ edges has value $m_2/(z+2)$, and thus all edges of $\Delta_{m_1(n_1),m_2}$ have value at least $m_2/(z+2)$, except edges in the symmetry breaking ring.

We show that $\Delta_{m_1(n_1),m_2}$ has $\frac{m_1}{z+1} + (1+o_{m_1}(1))\alpha \beta^z z^{-7/2}\left(\frac{\beta}{\beta-1}\right)^2$ components of value $m_2$ and $o(\beta^z z^{-7/2}) = o(m_2/z^2)$ components of lesser value.  The latter statement follows by Lemma \ref{n1is2lemma}.  Thus in fact $\Delta_{m_1(n_1),m_2}$ contains $(1+o_i(1))\alpha' \beta^i i^{-3/2}$ defect-free components on $i$ edges for $1 \leq i \leq z$, $o(m_1/z)$ defect-free components on $z+2$ edges, $o(m_1/z^2)$ components of other types, and all remaining components are defect-free on $z+1$ edges.

For every edge $e \in E(\Delta_{m_1(n_1),m_2})$, let $v^*(e) = v(e) - m_2/(z+1)$.  For $1 \leq i \leq z$, the sum of $v^*(e)$ over all edges $e$ in components with $z+1-i$ edges is thus $$(m_2+o(1))\alpha \beta^{z+1-i} (z+1-i)^{-3/2}(1/(z+1-i)-1/(z+1)).$$  By $\beta^z z^{-7/2} = \Theta(m_1 / \log^2 m_1)$, the preceding sum is $$m_2\alpha \beta^{z+1-i}z^{-3/2}(i/z^2) + o(m_1\beta^{-i}/\log^2 m_1)$$ for $i < z / \log z$, and otherwise $$m_2\alpha \beta^{z+1-i}z^{-3/2}(i/z^2) + o(m_1/\log^3 m_1),$$ which is observed by noting that $\beta^{z-z/\log z} = O((m_2 \log(m_1)^{3/2})^{1-1 / \log z})$, and that $m_1^{1 / \log z} > m_1^{i \log \log m_1 / \log m_1} = \log^i(m_1)$ for all fixed $i$.

The sum of $v^*(e)$ over all edges $e$ in components with either $z+2$ edges or with defects is $o(m_1/\log^2(m_1))$, whereas $v^*(e) = 0$ if $e$ is in a defect-free component with $z+1$ edges.  We conclude that $$\sum_{e \in E(\tau'(P))}v^*(e) = \alpha \beta^z z^{-7/2}\sum_{i=0}^z \beta^{-i}(i+1) + o(m_1 / \log^2 m_1)$$ $$= \alpha \beta^z z^{-7/2} \frac{\beta^2}{(\beta-1)^2} + o(m_1 / \log^2 m_1).$$

Thus $$\sum_{e \in E(\tau'(P))}v(e) = \frac{m_1}{z+1} + \alpha \beta^z z^{-7/2} \frac{\beta^2}{(\beta-1)^2} + o(m_1 / \log^2 m_1),$$ which implies that $G$ has the desired number of components of value $m_2$.
\endproof

The proof of Lemma \ref{n1is2lemma} makes use of the species $L(\mathfrak{a^\cdot})$ of ordered sets of rooted asymmetric trees: an element of $L(\mathfrak{a^\cdot})$ on $z'$ elements is given by order partitioning $[z']$ into subsets and taking an $\mathfrak{a}^\cdot$-structure on each subset.

\begin{lemma}
\label{n1is2lemma}
There are $o(m_1/z^2)$ components of $\Delta_{m_1(n_1),m_2}$ with defects.
\end{lemma}

\proof
If $G$ is a component of $\Delta_{m_1(n_1),m_2}$ with a defect, then the value of $G$ is at most $m_2-1$, and thus since the edges of $G$ have value at least $m_2/(z+2)$, then $G$ has at most $\frac{m_2-1}{m_2}(z+2)$ edges.  Thus we need to show that the number of components with defects on at most $\frac{m_2-1}{m_2}(z+2)$ edges in $\Delta_{m_1(n_1),m_2}$ is $o(\beta^z z^{-9/2})$.  It suffices to show that the number of components with positive value on $z'$ edges is $O((\beta')^{z'})$ for fixed $\beta < \beta' < \beta^{m_2/(m_2-1)}$ and $z' < \frac{m_2-1}{m_2}(z+2)$, since $\beta^{z'} < \beta^z / z^i$ for fixed $i$.

If $G$ is a component of positive value with $z'$ edges, then by Lemma \ref{treevalue} all but at most $m_2-1$ vertices of $G$ are labelled $[m_2]$.  Let $G'$ be the subgraph of $G$ that is the union of all paths between vertices not labelled $[m_2]$.  Only vertices not labelled $[m_2]$ are leaves in $G'$, and each leaf has one of $2^{m_2}-1$ labels, and thus the number of such $G'$ that may result is at most a polynomial in $z'$, say $p(z')$.  We may reconstruct $G$ from $G'$ by replacing every vertex $v \in G'$ with a rooted asymmetric tree with root $v$.  Thus, since $G$ is determined by $G'$ and an ordered set of asymmetric trees on a total of $z'+1$ vertices, there are at most $p(z')|L(\mathfrak{a^\cdot})_{z'+1}|$ components on $z'$ edges.

Choose fixed $\beta < \beta^* < \beta'$.  The lemma follows by showing that $|L(\mathfrak{a^\cdot})_{z'+1}| = O((\beta^*)^{z'})$.  We show inductively on $z'$ that $|L(\mathfrak{a^\cdot})_{z'+1}| \leq \gamma(\beta^*)^{z'}$ for some sufficiently large $\gamma$.

Note the recursion $L(\mathfrak{a^\cdot}) = E_0 + \mathfrak{a}^\cdot L(\mathfrak{a^\cdot})$: every ordered set of rooted trees is either the empty set or a rooted tree followed by another ordered set of rooted trees.  Thus for $z' \geq 1$, $$|L(\mathfrak{a^\cdot})_{z'+1}| = \sum_{i=1}^{z'+1} |\mathfrak{a_i^{\cdot}}||L(\mathfrak{a^\cdot})_{z'+1-i}| \leq \sum_{i=1}^{z'+1} |\mathfrak{a}^{\cdot}_i| \gamma (\beta^*)^{z'+1-i}.$$

Let $\mathfrak{a}^\cdot_{\leq z'/3}$ be the species of rooted asymmetric trees on at most $z'/3$ vertices.  Observe that $|(\mathfrak{a}^\cdot_{\leq z'/3} \mathfrak{a}^\cdot)_{z'+1}| \leq |\mathfrak{a}^\cdot_{z'+1}|$: given an asymmetric rooted tree $T$ on $i \leq z'/3$ vertices and another $T'$ on $z'+1-i$ vertices, a third tree may be constructed by adjoining $T'$ to $T$ so that the root of $T'$ is forgotten and place adjacent to the root of $T$.  This construction allows $T$ and $T'$ to be uniquely determined.  Thus by Lemma \ref{ordinarytreecount}, $$\sum_{i=1}^{\lfloor z'/3 \rfloor} |\mathfrak{a}^\cdot_i| |\mathfrak{a}^\cdot_{z'+1-i}| \leq (1+o(1))\alpha' \beta \beta^{z'} z'^{-3/2}.$$  But also, $$\sum_{i=1}^{\lfloor z'/3 \rfloor} |\mathfrak{a}^\cdot_i| |\mathfrak{a}^\cdot_{z'+1-i}| \geq \alpha' \beta \beta^{z'} z'^{-3/2}(1+o(1))\sum_{i=1}^{\lfloor z'/3 \rfloor} |\mathfrak{a}^\cdot_i |\beta^{-i}.$$ Thus $$\sum_{i=1}^{\lfloor z'/3 \rfloor} |\mathfrak{a}^\cdot_i |\beta^{-i} \leq 1+o(1) \quad \mbox{and} \quad \sum_{i=1}^{\lfloor z'/3 \rfloor} |\mathfrak{a}^\cdot_i |{(\beta^*)}^{-i} < \beta/\beta^*.$$

Thus $\sum_{i=1}^{z'+1} |\mathfrak{a}_i^{\cdot}| \gamma (\beta^*)^{z'+1-i} \leq$ $$\gamma \beta(\beta^*)^{z'} + \sum_{i = \lfloor z'/3 \rfloor + 1}^{z'+1} |\mathfrak{a}_i| \gamma (\beta^*)^{z'+1-i} = $$ $$\gamma \beta (\beta^*)^{z'} + (1+o(1))\sum_{i = \lfloor z'/3 \rfloor + 1}^{z'+1} \alpha' \beta^i i^{-3/2}\gamma (\beta^*)^{z'+1-i} < \gamma (\beta^*)^{z'+1}$$ as desired.
\endproof

\section{$k=0$ and $j < \lfloor (m_2-1)/2 \rfloor$}
\label{kis0jsmall}

In this section we consider the case that $k = 0$ and $j < \lfloor (m_2-1)/2 \rfloor$.  The value of $C$ defined in the statement of Theorem \ref{kis0jsmallthm} is determined as follows.  If $0 < j < m_2/2 - 3/2,$ then

$$C = {m_2 \choose j+1}^{m_2-2j-1}{m_2 \choose j}^{m_2-2j-3}{2m_2 - 4j - 4 \choose m_2 -2j - 3}\frac{2^{-m_2+2j+3}}{(2m_2-4j-4)!}.$$
If $j= 0$ and $m_2 > 3$, then
$$C = m_2^{m_2-1}{2m_2 - 4 \choose m_2 - 3}\frac{1}{(2m_2 - 4)!},$$ and if $j = m_2/2 - 3/2$, then $C = {m_2 \choose j+1}\left({m_2 \choose j+1}-1\right)/2$.

The result requires the following estimate on the number of structures of a particular type of tree.

\begin{lemma}
\label{kis0treecountlemma}
If $i$ is odd, the number of labelled structres of $\mathfrak{A}_{\mathcal{E}_1+\mathcal{E}_3}$ on $i+1$ vertices is ${i+1 \choose (i-1)/2} \frac{(i-1)!}{2^{(i-1)/2}}$, and of $\mathfrak{A}_{\mathcal{E}_1+2\mathcal{E}_3}$ on $i+1$ vertices is ${i+1 \choose (i-1)/2} (i-1)!$.
\end{lemma}
\begin{proof}
Apply Proposition 3.1.19 of Bergeron, Labelle, and Leroux \cite{species}.
\end{proof}

\proofof{Theorem \ref{kis0jsmallthm}}
We start by proving an upper bound on the sum of the values of all components of $\Delta_{m_1(n_1),m_2}$.  Let $G$ be a component of $\Delta_{m_1(n_1),m_2}$ with $t$ edges and positive value; by Lemma \ref{treevalue}, $G$ is a tree with at most $m_2-2j-1$ leaves.

Construct a colored graph $c(G)$ from $G$ as follows: $V(c(G))$ is the union of all edges of $G$ and vertices of $G$ of degree at least two; and $E(c(G))$ is given by vertex-edge containment.  Every defective vertex or edge in $G$ is colored red in $c(G)$.  If $v$ is a non-defective vertex of degree at least $3$, then $v$ is colored green in $c(G)$.  All other vertices of $c(G)$ are blue.

A \textit{segment} in $c(G)$ is a maximal path $(v_0, \ldots, v_i)$ such that for all $0 < i' < i$, $v_{i'}$ is a blue vertex with degree $2$.  Construct a new graph $c'(G)$ by replacing every segment $(v_0, \ldots, v_i)$ with a single edge $v_0v_i$, and label that edge by the number $i$ of edges it replaces in $c(G)$.

The number of edges of $c'(G)$ is at most $2m_2 - 4j - 5$.  To see this, observe that $G$ has $d$ defects and at most $m_2-2j-1-d$ leaves by Lemma \ref{treevalue}.  Every leaf of $c'(G)$ is a leaf of $G$.  Combining the facts that $\sum_{v \in V(c'(G))}\deg(v) = 2e(c'(G))$ and $\sum_{v \in V(c'(G))}(\deg(v)-2) = -2$, $e(c'(G)) \leq 2m_2-4j-5-2d+a-b$, where $a$ and $b$ are the number of vertices of degree $2$ and at least $4$ in $c'(G)$.  However, the only degree $2$ vertices in $c'(G)$ correspond to defects in $G$, and thus $a \leq d$.  Thus $c'(G)$ has at most $2m_2 - 4j  - 5$ edges.  Furthermore, this bound is attained only if $G$ has $m_2-2j-1$ leaves, no defects, no vertices of degree at least $4$, and no vertices of degree $3$ if $j>0$ by Lemma \ref{treevalue}.

If $c'(G)$ has $2m_2-4j-5$ edges, then $G$ has value $1$ and no edges that intersects four other edges, since this would give a vertex of degree at least $4$ in $c'(G)$.  Thus, if $c'(G)$ has $2m_2-4j-5$ edges, then $c'(G)$ has $m_2-2j-1$ leaves and $m_2-2j-3$ vertices of degree $3$.  If $j>0$, then by Lemma \ref{treevalue}, all vertices of $c'(G)$ are blue, and such trees, forgetting labels, may be described by the species $\mathfrak{A}_{\mathcal{E}_1+\mathcal{E}_3}$.  If $j=0$, then the degree $3$ vertices may be green or blue, and thus such trees are described by the species $\mathfrak{A}_{\mathcal{E}_1+2\mathcal{E}_3}$.

Given $c'(G)$ with $i=2m_2-4j-5 \geq 2$ edges and that $G$ has $t$ edges, there are $(1+o_1(t))t^{i-1}/((i-1)!\alpha)$ nonisomorphic labelings of the edges, where $\alpha$ is the cardinality of the automorphism group of $c'(G)$.  This is since that in most labelings, all labels are distinct, and the orbit of a labeling with distinct labels consists of $\alpha$ labelings. The number of  labelled graphs $c'(G)$ of a given isomorphism class and automorphism group of order $\alpha$ is $(i+1)!/\alpha$.  Hence the number of graphs $c(G)$ with $c'(G)$ having $m_2-2j-1$ leaves is $\gamma(1+o_1(t))t^{i-1}/((i-1)!(i+1)!)$, where $\gamma$ is the number of labeled specimens of $\mathfrak{A}_R$ as in Lemma \ref{kis0treecountlemma}.

If $c'(G)$ has fewer than $i=2m_2-4j-5$ edges, there are $o(t^{i-1})$ labelings of $c'(G)$.  Since the number of graphs $c'(G)$ that may arise is independent of $t$, the total number of graphs $c(G)$ is $\gamma(1+o_1(t))t^{i-1}/((i-1)!(i+1)!)$, of which almost all have $m_2-2j-1$ leaves and all segments of different lengths.

Given a graph $G'$, the number of components $G$ of $\Delta_{m_1(n_1),m_2}$ with positive value such that $c(G) = G'$ has an upper bound that depends only on $n_1$ and $m_2$.  $G$ is determined by $c(G)$ and the following: the labels of all defective vertices, the labels of all vertices that are contained in defective edges, the labels of all vertices contained in leaves of $G$, and the labels of all vertices contained in edges with at least $3$ vertices of degree at least $2$.  There are at most $m_2 - 2j - 1$ of each of these items in $G$, and each one may be determined in at most $2^{n_1m_2}$ ways, and thus there are at most $2^{4n_1m_2(m_2-2j-1)}$ components $G$ with positive value such that $c(G) = G'$.  Thus, there are $o(t^{2m_2-4j-6})$ positive-value components $G$ with $t$ edges such that either $G$ has at most $m_2-2j-2$ leaves or $c'(G)$ has two edges with the same label.  Adding over all $t$, there are $o(m_1^{\frac{2m_2-4j-5}{2m_2 - 4j - 4}})$ such components $G$.

Now we determine how many positive-value components $G$ with $t$ edges of $\Delta_{m_1(n_1),m_2}$ satisfy these two conditions: $c(G) = G'$ for a particular graph $G'$ with $m_2-2j-1$ leaves, and $c'(G)$ has distinct edge labels.  If $0 \leq j < m_2/2 - 3/2$, $G$ has no defects, no vertices of degree at least $3$ (if $j>0$), and $m_2-2j-3$ edges that intersect $3$ others.  Each leaf, since it is not defective, contains one vertex of every label $S$ with $|S| \geq m_2-j$ and exactly one vertex with a label $S$ with $|S| = m_2-j-1$.  There are ${m_2 \choose j+1}$ ways to select this label.  Each edges that intersects $3$ others, since it is not defective, contains a vertex of every label $S$ with $|S| \geq m_2-j$ except for one label $S$ with $|S| = m_2-j$.  There are ${m_2 \choose j}$ ways to choose this label.  The total number of such components is ${m_2 \choose j+1}^{m_2-2j-1}{m_2 \choose j}^{m_2-2j-3}$, and the distinct edge labels of $c'(G)$ ensure that each of these components are asymmetric.

If $j = m_2/2 - 3/2$, then $c(G)$ has $2$ leaves and is a path.  As before, the two leaves each contain a vertex of every label $S$ with $|S| \geq m_2-j$, together with one vertex each of labels $S$ and $S'$ respectively with $|S| = |S'| = m_2-j-1$.  All other edges contain exactly a vertex of each label $\tilde{S}$ with $|\tilde{S}| \geq m_2-j$.  By asymmetry, $S \neq S'$.  Thus there are ${m_2 \choose j+1}({m_2 \choose j+1}-1)/2$ asymmetric components $G$ with $c(G) = G'$.

We conclude that there are $(C+o(1))t^{2m_2-4j-6}$ components of $\Delta_{m_1(n_1),m_2}$ with $t$ edges and positive value, almost all of which have value $1$ and none with value exceeding $m_2-2j-2$.  Adding over all $$t < (1+o(1))\left(\frac{m_1(2m_2-4j-4)}{C}\right)^{\frac{1}{2m_2-4j-4}}$$ proves the result.

\endproof

\section{$k \geq 1$ and $j < \lfloor (m_2-1)/2 \rfloor$}
\label{othersection}

\proofof{Theorem \ref{otherthm}}
We start with the upper bound on $n_2$.  By Lemma \ref{treevalue}, every component of $\Delta_{m_1(n_1),m_2}$ has value at most $m_2-1$, and every component with positive value is a tree.  Suppose that the number of components of $\Delta_{m_1(n_1),m_2}$ on $i$ edges with positive value is at most $b^i$ for some $b$.  Then $\Delta_{m_1(n_1),m_2}$ has at most $\frac{b^{\lceil (\log_b(m_1))/2 \rceil +1}-1}{b-1}$ components with at most $\lceil \log_b (m_1) / 2 \rceil $ edges.  Thus, $\Delta_{m_1(n_1),m_2}$ has at most $\frac{b^{\lceil (\log_b(m_1))/2 \rceil +1}-1}{b-1} + \frac{m_1}{\lceil \log_b(m_1)/2 \rceil}$ components of positive value, each of which has value at most $m_2-1$.  This would prove the upper bound.

We now establish that there are at most $b^i$ components of positive value on $i$ edges for some $b$.  Every component $G$ of positive value is a tree.  Associate with $G$ a labelled tree $G'$ as follows.  The vertex set of $G'$ is the union of the edge set of $G$ and the set of vertices of $G$ with degree at least $2$.  The edges of $G'$ are given by inclusion in $G$.  If $v$ is a vertex of $G'$ that corresponds to a vertex of $G$, then $v$ is given the same label; thus, there are at most $2^{m_2}$ possible labels for $v$.  If $e$ is a vertex of $G'$ that corresponds to an edge of $G$, then $e$ is labelled in a way to encode the number and labels of degree $1$ vertices of $e$.  Thus $e$ can be labelled in at most $1 + 2^{m_1} + 2^{2m_1} + \cdots + 2^{n_2m_1}$ ways.  $G$ can be reconstructed to isomorphism from $G'$.

If $G$ has $i$ edges, then $G'$ has at most $2i-1$ vertices.  Thus the number of isomorphism classes of underlying unlabelled trees of $G'$ grows exponentially in $i$ \cite{otter}.  Since the number of possible labels of each vertex of $G'$ depends only on $n_1$ and $m_2$, the total number of trees $G'$, and thus components $G$, grows at most exponentially in $i$.

To prove the lower bound on $n_2$, we show that the number of components of $\Delta_{m_1(n_1),m_2}$ with $t$ edges does in fact grow exponentially in $t$.  Let $G$ be a tree without defects or vertices of degree at least $3$ such that every edge contains at least ${m_2 \choose 0} + \cdots + {m_2 \choose j}$ degree $1$ vertices.  Then $G$ contains $2|E(G)|-1$ vertices of weight $m_2$, $|E(G)|{m_2 \choose i}$ of weight $m_2-i$ for $1 \leq i \leq j$, and $|E(G)|k + 2$ of weight $m_2-j-1$.  Then $G$ has value $m_2-2j-2$.

To $G$ we may associate a vertex-labelled tree $G'$, with the vertices of $G'$ given by the edges of $G$, and the edges of $G'$ are given by intersection.  The label of a vertex of $G'$ encodes the labels of the degree $1$ vertices of the corresponding edge.  Suppose that such an edge $e$ intersects $i$ other edges.  Since $e$ contains a vertex of every label with of size at least $m_2-j$ and $k-i+2$ vertices of label of size $m_2-j-1$, there are $a_i := {{m_2 \choose j+1} \choose k-i+2}$ ways to label $e$ in $G'$.  Thus $G'$ may be regarded as a member of $\mathfrak{a}_{\sum_{i=1}^{k+2} a_i\mathcal{E}_i}$, and from every member of this species, one can reconstruct an asymmetric $2^{[m_2]}$-labelled $n_1$-uniform tree with positive value.  We show that $|(\mathfrak{a}_{\sum_{i=1}^{k+2} a_i\mathcal{E}_i})_t|$ grows exponentially in $t$ by exhibiting a subset of structures of exponential size.

Let $(v_0, \ldots, v_{\lfloor 2t/3 \rfloor})$ be a path.  Let $S$ be subset of size $t-1-\lfloor 2t/3 \rfloor$ of the integers from $3$ to $\lfloor 2t/3 \rfloor-2$ that includes $3$ and $\lfloor 2t/3 \rfloor-2$.  For every $i \in S$, let $u_i$ be a vertex with an edge $u_iv_i$.  Then the graph with vertices $(v_0, \ldots, v_{\lfloor 2t/3 \rfloor})$ and $u_i$ for each $i \in S$, and labels chosen arbitrarily, is an element of $\mathfrak{a}_{\sum_{i=1}^{k+2} a_i\mathcal{E}_i}$.  Two such graphs are nonisomorphic for different choices of $S$, and the number of choices of $S$ grows exponentially in $t$.

Say that there are $b^t$ trees of the maximum possible value of $m_2-2j-2$ on $t$ edges for some fixed $b$ and sufficiently large $t$.  Then $\Delta_{m_1(n_1),m_2}$ contains no component with more than $\lceil \log_b(m_1) \rceil$ edges for large $m_1$, except possibly the symmetry breaking ring, and $\Delta_{m_1(n_1),m_2}$ has at least $\frac{m_1-m_2\lceil \log_b(m_1) \rceil-\min_R}{\lceil \log_b(m_1)\rceil}$ components.  This proves the theorem.
\endproof

\section{$k=0$ and $j = \lfloor (m_2-1)/2 \rfloor$}
\label{kis0jlarge}

We consider Theorem \ref{kis0jlargethm} in three cases.

\begin{theorem}
Theorem \ref{kis0jlargethm} holds for even $m_2 \geq 4$.
\end{theorem}

\proof
The upper bound on $n_2$ follows from Lemmas \ref{treevalue} and \ref{valuetotal}: when $j = (m_2-1)/2$, since every tree has at least $2$ leaves, no component has positive value.

We establish the that $n_2$ may be $rm_1+2^{m_2-1}$ by the following construction.  Let $\tau'(P)$ consist of components $G_1, \ldots, G_{m_2}$ of $n_1$-edges, such that $G_i$ consists of $t_i$ edges, and all the $t_i$ are distinct and sum to $m_1$. Say that $G_i$ contains edges $e_1, \ldots, e_{t_i}$ such that for all $1 \leq a < b \leq t_i$, $e_a$ and $e_b$ do not intersect unless $b = a+1$, in which case $e_a \cap e_b = \{v_a\}$.  Each $v_a$ is labelled $[m_2]$.  Each $e_a$ contains one vertex of each of label of size at least $m_2-j$.  In addition, $e_1$ and $e_{t_i}$ contain respective vertices $u_1$ and $u_2$ of labels $\{i,i+1, \ldots, i+j\}$ and $\{i, i-1, \ldots, i-j\}$, subscripts mod $m_2$.  In addition, $\tau'(P)$ contains one degree $0$ vertex of each nonempty label.

Now we show that $\tau(P)$ is asymmetric.  Since the $G_i$ have different numbers of edges, no automorphism permutes the components of $\tau'(P)$ nontrivially.  The only nontrivial automorphism of the edges of $G_i$ reverses the chain.  Given that an automorphism $\sigma$ fixes the $n_2$-edges of $\tau(P)$, the two leaf edges of $G_i$ cannot be interchanged, and thus are fixed and all edges of $G_i$ are fixed.  Thus every degree $2$ vertex in $\tau'(P)$ is fixed as well.  Finally, each degree $1$ vertex in an edge $e \in H_i$ has a different label, and thus all these vertices are fixed.

Finally, since $\sigma$ fixes each $G_i$ componentwise, and the $n_2$-edge $X_i$ intersects $G_i$ more than any other $n_2$-edge, $\sigma$ fixes each $n_2$-edge.
\endproof

\begin{theorem}
Theorem \ref{kis0jlargethm} holds when $m_2 = 2$.
\end{theorem}

\begin{proof}
All components of $\tau'(P)$ have a a nonpositive value by Lemma \ref{treevalue}.  By Lemma \ref{treevalue}, if $G$ is a tree of $0$ value, then $G$ has two leaves and is a chain.  Furthermore, to assure asymmetry, the leaves of $G$ must contain vertices labelled $\{1\}$ and $\{2\}$ respectively.  Otherwise, if $G$ is a non-tree with value $0$, then since $r=2$, $G$ must have $t$ edges and $2t$ vertices, and every vertex must be labelled $[2]$.

We conclude that if $\tau'(P)$ has value $m_2 2^{m_2-1}$, then $\tau'(P)$ is a collection of chains, as described above; components in which every vertex is labelled $[2]$; and a degree $0$ vertex of every nonempty label.  But then $\tau(P)$ has a symmetry that results from reversing each chain and switching the $n_2$-edges.  Thus the upper bound on $n_2$ holds.

Now we prove the sufficiency of the bound by construction.  Let $\tau'(P)$ contain a chain with at least five edges, $e_1, \ldots, e_{m_1}$ such that $e_a$ and $e_b$ intersect only when $b = a+1$, and then $e_a \cap e_b = \{v_a\}$.  All vertices are labelled $[2]$ except for $v_1$ and $v_2$, which are labelled $\{1\}$ and $\{2\}$ respectively, and degree $1$ vertices $u_1$ and $u_2$ that are contained in each of the leaves, labelled $\{1\}$ and $\{2\}$ respectively.  Also, $\tau'(P)$ contains a degree $0$ vertex of each nonempty label.

We show that $\tau(P)$ is asymmetric.  The $n_2$-edges cannot be switched since no automorphism of $\tau'(P)$ moves $v_1$ to any other vertex of weight $1$.  Thus $v_1$ is a fixed point in $\tau(P)$, which fixes all edges of $\tau'(P)$.  Furthermore, the vertices in the leaves of $\tau'(P)$ are fixed since they are of different labels.
\end{proof}

\begin{theorem}
Theorem \ref{kis0jlargethm} holds for odd $m_2$.
\end{theorem}
\proof
By Lemma \ref{treevalue}, no component of $\tau'(P)$ has positive value.  We establish the upper bound on $n_2$ by showing that every connected component of $\tau'(P)$ has negative value.

Suppose that $G$ is a component with value $0$.  By Lemma \ref{weightbound}, $G$ contains $|E(G)|(n_1-1)$ vertices, of which $|E(G)|(n_1-2)$ have degree $1$.  Furthermore, every edge contains $n_1-2$ degree $1$ vertices since $w_{n_1-1} + w_{n_1-3} < 2w_{n_2-2}$.  Construct $G'$ by removing these degree $1$ vertices.  Then $G'$ is a ordinary, $2$-regular connected graph and thus a cycle.  Furthermore, $G$ lacks defects.  We conclude that $G$ has a nontrivial automorphism.

Now consider the following construction when $m_1 \geq m_2 + 3$.  Let $\tau'(P)$ contain a connected $n_1$-uniform hypergraph with edges $e_1, \ldots, e_{m_1}$, subscripts mod $m_1$, such that $e_a$ and $e_b$ do not intersect unless $|b-a| = 1$.  If $b = a+1$, then we say that $e_a \cap e_b = \{v_a\}$.  $\tau'(P)$ has no defects except for the following.  For $1 \leq i \leq m_2 - 1$, say that $v_i$ is labelled $[m_2]-\{i\}$, and $v_{m_2+1}$ is labelled $[m_2]-\{m_2\}$.  Also, $\tau'(P)$ contains a degree $0$ vertex of each nonempty label.  Then $\tau'(P)$ is asymmetric, and $n_2$ is the maximum value.
\endproof

\section{$k \geq 1$ and $j = \lfloor (m_2-1)/2 \rfloor$}
\label{k1jlarge}

We say $J$ that an $(\psi,s,t,n_1)$-regular hypergraph if $J$ satisfies the following conditions.  $J$ is $n_1$-uniform with $t$ edges, of which all edges intersect $s$ other edges with the exception of $\psi$ edges that each intersect $s-1$ other edges.  Furthermore, no two edges intersect at more than $1$ vertex, and every vertex has degree at most $2$.  Finally, every automorphism of $J$ fixes all edges.  Before we prove the main result of this section, we need some lemmas on the existence of regular hypergraphs.

\begin{lemma}
\label{graphexistence2}
Let $\psi$ and $s \geq 3$ be given, and suppose that $t$ is sufficiently large relative to $\psi$ and $s$.  Suppose that $st-\psi$ is even.  Then there exists an asymmetric graph with $t$ vertices such that all vertices have degree $s$, except for $\psi$ vertices that have degree $s-1$ .
\end{lemma}

Call a graph of this form an $(\psi,s,t)$-\textit{asymmetric} graph.

\begin{proof}
The lemma follows from the main theorem of \cite{randomregular} when $st$ is even and $\psi=0$, since an random $s$-regular graph is almost surely asymmetric.  Such a graph is also almost surely $s$-connected \cite{boll}.

Consider the case that $st$ is even.  Then $\psi$ is also even.  Choose distinct $t_1, \ldots, t_{\psi/2}$ such that $t_1 + \ldots + t_{\psi/2} = t$, with each $t_i$ even if $t$ is even.  For $1 \leq i \leq \psi/2$, let $G_i$ be an $s$-connected $(0,s,t_i)$-asymmetric graph with an edge removed.  Then the disjoint union of the $G_i$ is an $(\psi,s,t)$-asymmetric graph.

For odd $st$ and $\psi$, we may construct a $(\psi,s,t)$-asymmetric graph as follows.  Let $G'$ be an $(\psi(s-1),s,t-\psi)$-asymmetric graph.  Add new vertices $v_1, \ldots, v_{\psi}$ to $G'$, each with disjoint neighbor sets of size $s-1$ vertices of degree $s-1$ in $G'$.  The resulting graph is $(\psi,s,t)$-asymmetric.
\end{proof}

\begin{lemma}
\label{hgraphex}
Let $\psi$ and $s \geq 3$ be given, and suppose that $t$ is sufficiently large relative to $\psi$ and $s$.  Suppose that $st-\psi$ is even.  Also let $n_1 \geq s$ be given.  Then there exists an $(\psi,s,t,n_1)$-asymmetric hypergraph.
\end{lemma}
\begin{proof}
Let $G$ be an $(\psi,s,t)$-asymmetric graph.  Let $H'$ be a hypgergraph with vertex set $E(G)$, edge set $V(G)$, and incidence given by indicent in $G$.  Then construct $H$ from $H'$ by adding $n_1-d$ degree $1$ vertices to every edge in $H'$ that contains $d$ vertices.  Then $H$ is $(\psi,s,t,n_1)$-asymmetric.
\end{proof}

\proofof{Theorem \ref{k1jlargethm}}
The upper bound on $n_2$ follows by Corollary \ref{valuecor} and \ref{valuetotal}, except when $km_1$ is even and $m_2$ odd.  In this case, suppose that all connected components of $\tau'(P)$ have value $0$, and $\tau'(P)$ contains a degree $0$ vertex of every nonempty label.  By Lemma \ref{weightbound} and the fact that $w_{n_1-k-1} + w_{n_1-k-3} < 2w_{n_1-k-2}$, every edge of $\tau'(P)$ contains exactly $n_1-k-2$ degree $1$ vertices.  Furthermore, $\tau'(P)$ does not have any defects.  In this case, $\tau(P)$ has an nontrivial automorphism that permutes the $n_2$-edges.  The upper bound on $n_2$ follows.

We establish the result by the following constructions.  First consider the case that $km_1$ and $m_2$ are both even.  Let the graph of $\tau'(P)$ be an $(m_2, k+2, m_1, n_1)$-asymmetric hypergraph, which exists by Lemma \ref{hgraphex}, together with a degree $0$ vertex of every nonempty label.  Choose the labels of the vertices of $\tau'(P)$ so that there are no defects, label the edges with $k+1$ degree $1$ vertices by $e_1, \ldots, e_{m_2}$, and say $e_i$ contains a vertex with label $\{i,i+1, \ldots, i+m_2/2-1\}$, subscripts mod $m_2$.  Then $\tau'(P)$ is asymmetric and satisfies $n_2 = rm_1 + 2^{m_2-1}$.

If $km_1$ is even and $m_2$ is odd, let $\tau'(P)$ be a $(0,k+2,m_1,n_1)$-asymmetric hypergraph, together with a degree $0$ vertex of every nonempty label.  Suppose that $\tau'(P)$ has exactly the following $m_2$ defects: for degree $2$ vertices $v_1, \ldots, v_{m_2}$, $v_i$ is labelled $[m_2]-\{i\}$.  Then $\tau(P)$ is asymmetric, and $n_2 = rm_1 + 2^{m_2-1} - 1$.

If $km_1$ is odd and $m_2$ is even, then let $\tau'(P)$ be an $(m_2+1,k+2,m_1,n_1)$-asymmetric hypergraph with $e_1, \ldots, e_{m_2+1}$ the edges that intersect $k+1$ other edges, together with a degree $0$ vertex of every nonempty label.  Choose the labels of the vertices of $\tau'(P)$ so that there are no defects, except that $e_{m_2+1}$ contains a degree $1$ vertex with every label of size at least $m_2-j$, together with a vertex labelled $\emptyset$.  For $1 \leq i \leq m_2$, $e_i$ contains a degree $1$ vertex labelled $\{i,i+1, \ldots, i+m_2/2-1\}$, subscripts mod $m_2$.  Then $\tau(P)$ is asymmetric and satisfies $n_2 = rm_1 + 2^{m_2-1} - 1/2$.

Finally, if $km_1$ and $m_2$ are both odd, then let $\tau'(P)$ be an $(m_2,k+2,m_1,n_1)$-asymmetric hypergraph, together with a degree $0$ vertex of every nonempty label.  Choose the labels of $\tau'(P)$ so that there are no defects, and if the edges with $k+1$ degree $1$ vertices are $e_1, \ldots, e_{m_2}$, then $e_i$ contains a vertex labelled $\{i,i+1,\ldots,i+m_2/2-1/2\}$, with subscripts mod $m_2$.  Then $\tau(P)$ is asymmetric and satisfies $n_2 = rm_1 + 2^{m_2-1} - 1/2$.
\endproof

\end{document}